\documentclass{amsart}

\usepackage{times}
\usepackage{mathtools}
\usepackage{geometry}
\usepackage{amssymb}
\usepackage{amsmath}
\usepackage{amsthm}
\usepackage{dsfont}
\usepackage{enumerate}
\usepackage{enumitem}
\usepackage{tikz}
\usepackage{tikz-cd}

\theoremstyle{plain}
\newtheorem{theorem}{Theorem}

\newtheorem*{Ack}{Acknowledgements}

\theoremstyle{definition}

\newtheorem{remark}[theorem]{Remark}

\newtheorem{example}[theorem]{Example}

\theoremstyle{remark}

\title{Frobenius elements in Galois representations with $\SL_n$ image}
\author{Matthew Bisatt}
\address{Department of Natural and Mathematical Sciences, Strand Campus, King's College London, London, WC2R 2LS, United Kingdom}
\email{matthew.bisatt@kcl.ac.uk}

\date{\today}

\begin{document}
\global\long\def\fl{\mathbb{F}_l}
\global\long\def\qq{\mathbb{Q}}
\global\long\def\cc{\mathbb{C}}
\global\long\def\zz{\mathbb{Z}}
\global\long\def\Gal{\operatorname{Gal}}
\global\long\def\Im{\operatorname{Im}}
\global\long\def\Frob{\operatorname{Frob}}
\global\long\def\SL{\operatorname{SL}}
\global\long\def\GL{\operatorname{GL}}


\maketitle

Suppose $F$ is a number field and $f \in F[x]$ is an irreducible polynomial of degree $n$ with splitting field $K$. How can we determine the $\Gal(K/F)$-conjugacy class of a Frobenius element, without explicitly constucting $K$? If we consider the Galois action on the roots, we can identify $\Gal(K/F)$ as a permutation group; the factorisation of $f$ over the residue field enables us to find the cycle type of Frobenius and hence we have it up to conjugacy in the symmetric group $S_n$. If the Galois group is isomorphic to the alternating group $A_5$ though, then this is insufficient if $f$ remains irreducible since there are two different conjugacy classes of $5$-cycles in $A_5$. Serre \cite[p.53]{Buhler1978} observed that computing a ``square root" of the discriminant of $f$ produced the extra necessary data. Roberts \cite{Roberts2004} then considered all alternating groups before Dokchitser and Dokchitser \cite{Dokchitser2013} generalised this to any finite group by constructing suitable resolvents.

In number theory, Galois extensions also arise from Galois representations; in this setting we have a natural linear action on the underlying vector space. We wish to incoporate this extra stucture so instead embed the Galois group into a matrix group as opposed to a permutation group to yield an alternative approach to the problem of distinguishing Frobenius elements. We shall not give a complete theory for differentiating conjugacy classes of an arbitrary matrix group, but consider an analogue of the $S_n$ versus $A_n$ situation: $\GL_n(\fl)$ versus $\SL_n(\fl)$. We illustrate our approach with the aid of elliptic curves, using the Weil pairing for our additional information. For the remainder of the paper, we shall abbreviate $\SL_2(\fl)$ and $\GL_2(\fl)$ to $\SL_2$ and $\GL_2$ respectively and say the $\GL_2$-conjugacy class of an element $\sigma \in \SL_2$ \emph{splits} if its $\SL_2$-conjugacy class is properly contained in its $\GL_2$-conjugacy class.

Let $E/F$ be an elliptic curve and fix a rational prime $l$. Then the action of $\Gal(\overline{F}/F)$ on the group $E[l]$ of $l$-torsion points gives rise to the mod $l$ Galois representation $$\rho_{E,l}: \Gal(\overline{F}/F) \rightarrow \GL_{2}(\fl),$$ which factors through $\Gal(K/F)$, where $K$ is the smallest extension of $F$ over which all $l$-torsion points are defined. Let $\mathfrak{p}$ be a prime of $F$ which is unramified in $K$ and doesn't divide the discriminant $\Delta_E$ of $E$ (it suffices to assume $\mathfrak{p} \nmid l\Delta_E$; this is the assumption we will generally use).

There are two standard pieces of information that we can acquire about the Frobenius element coming from its characteristic polynomial. Firstly, the determinant is equal to the absolute norm $q$ of $\mathfrak{p}$. We can also ascertain its trace by examining the number of points on the reduced curve (see for example Schoof's algorithm \cite{Schoof1985} or the refined Schoof-Elkies-Atkin algorithm). Unfortunately, these two pieces of data do not always completely distinguish the conjugacy class, even in $\GL_2$. 

\begin{remark}
	We shall suppose that $\Im \rho_{E,l}=\SL_2$, which implies that the $F$ contains a primitive $l^{th}$ root of unity $\zeta_l$ and $q \equiv 1 \mod{l}$. Recall that if $E/F$ does not have complex multiplication, then for all but finitely many rational primes $l$, $\zeta_l \in F$ implies that $\Im \rho_{E,l}=\SL_2$ by Serre's open image theorem.
\end{remark}

We now give a further criterion to distinguish between two classes in $\SL_2$ that are conjugate in $\GL_2$.

\begin{theorem}
\label{weilmethod}
	Let $E/F$ be an elliptic curve such that $\rho_{E,l}(\Gal(\overline{F}/F)) = \SL_2$ and $\mathfrak{p}$ be a prime of $F$ of absolute norm $q$ such that $\mathfrak{p} \nmid l\Delta_E$ and suppose that the $\GL_2$-conjugacy class of $\rho_{E,l}(\Frob_{\mathfrak{p}})$ splits.
	
	Let $\tilde{E}$ be the reduced curve at $\mathfrak{p}$ and suppose that $\{Q_1,Q_2\}$ is an ordered basis of $\tilde{E}[l]$ such that the action of the Frobenius automorphism $x \mapsto x^q$ acts as $\sigma \in \SL_2$ on $\tilde{E}[l]$ with respect to $\{Q_1,Q_2\}$.
	
	Then $\rho_{E,l}(\Frob_{\mathfrak{p}})$, written with respect to a global ordered basis $\{P_1,P_2\}$, is $\SL_2$-conjugate to $\sigma$ if and only if $$\langle P_1, P_2 \rangle_l \, \bmod{\mathfrak{p}} \equiv \langle Q_1, Q_2 \rangle_l^{k^2} \text{ for some } k \in \zz,$$ where $\langle \, , \, \rangle_l$ denotes the Weil pairing.	
\end{theorem}

\begin{proof}
	Write $\rho_{E,l}(\Frob_{\mathfrak{p}})=\tau$ (with respect to $P_1,P_2$) and observe that if $\tau=\sigma$ then the result trivially holds. By assumption, $\tau$ is $\GL_2$-conjugate to $\sigma$ so there exists $A \in \GL_2$ such that $\sigma=A^{-1}\tau A$. If $\det(A)=\lambda^2$ then let $A=\lambda A_1$ so $\det(A_1)=1$ and $\sigma=A_1^{-1}\tau A_1$ so they are $\SL_2$-conjugate. Conversely, since we have imposed that the $\GL_2$-conjugacy class splits, there exists $\sigma_1 \in \SL_2$ which is $\GL_2$-conjugate to $\sigma$, but not $\SL_2$-conjugate.
	
	Hence there exists $B \in \GL_2$ such that $\sigma=B^{-1}\sigma_1B$ and by the same argument, $\det(B)$ cannot be a square. If $\det A$ is not a square, then $\det(A)/\det(B)$ is a square so $\tau$ is $\SL_2$-conjugate to $\sigma_1$ therefore cannot be $\SL_2$-conjugate to $\sigma$.
	
	By the properties of the Weil pairing $\langle AP_1,AP_2 \rangle_l = \langle P_1,P_2 \rangle_l^{\det(A)}$, hence $\rho_{E,l}(\Frob_{\mathfrak{p}})$ is $\SL_2$-conjugate to $\sigma$ if and only if $\langle P_1, P_2 \rangle_l^{k^2} \mod{\mathfrak{p}} = \langle Q_1, Q_2 \rangle_l$ for some integer $k$ coprime to $l$. 
\end{proof}

We don't actually need the image to be $\SL_2$ to apply the above theorem. However, $F$ may not contain the relevant roots of unity so to combat this we should consider the minimal polynomials.

Let $m_F(\alpha)$ denote the minimal polynomial of $\alpha$ over $F$, for any field $F$ and algebraic number $\alpha$, and define $\GL_2^{\square}:=\{A \in \GL_2 \, | \, \det A \text{ is a square }\}$.

\begin{theorem}
	Let $E/F$ be an elliptic curve and suppose that $\rho_{E,l}(\Gal(\overline{F}/F)) = G \subset \GL_2^{\square}$. Let $\mathfrak{p}$ be a prime of norm $q$ such that $\mathfrak{p} \nmid l\Delta_E$ and $q \equiv 1 \mod{l}$. Let $\sigma \in \SL_2$ be $\GL_2$-conjugate to $\rho_{E,l}(\Frob_{\mathfrak{p}})$ and suppose that the $G$-conjugacy class of $\sigma$ is equal to the intersection of $G$ with its $\SL_2$-conjugacy class.
	
	Let $\tilde{E}$ be the reduced curve at $\mathfrak{p}$ and suppose that $\{Q_1,Q_2\}$ is an ordered basis of $\tilde{E}[l]$ such that the action of the Frobenius automorphism $x \mapsto x^q$ acts as $\sigma$ on $\tilde{E}[l]$ with respect to $\{Q_1,Q_2\}$.
	
	Then $\rho_{E,l}(\Frob_{\mathfrak{p}})$, written with respect to a global ordered basis $\{P_1,P_2\}$, is $G$-conjugate to $\sigma$ if and only if $$m_{\mathcal{F}}(\langle Q_1, Q_2 \rangle_l^{k^2}) \, \text{ divides } \, m_F(\langle P_1, P_2 \rangle_l) \, \bmod{\mathfrak{p}}$$ for some $k \in \zz$.	
\end{theorem}

\begin{example}
	Let $E/\qq(\zeta_3)$ be the elliptic curve $y^2=x^3+x+1$ (Cremona label 496a1), where $\zeta_3=e^{2\pi i/3}$. The image of the mod $3$ representation of $E/\qq(\zeta_3)$ is isomorphic to $\SL_2(\mathbb{F}_3)$. Let $\mathfrak{p}=(13,\zeta_3-3)$ be a prime of $\qq(\zeta_3)$. We shall compute the $\SL_2$-conjugacy class of $\rho_{E,3}(\Frob_{\mathfrak{p}})$.
	
	Choose a global basis $P_1=(\alpha_1,\beta_1), P_2=(\overline{\alpha_1},\overline{\beta_1}) \in E(\cc)$, where $\alpha_1 \approx 0.571 + 1.754i, \beta_1 \approx 0.984 + 2.761i$ and observe that $\langle P_1,P_2 \rangle_3=\zeta_3$.
	
	Now the reduced curve $\tilde{E}$ has $18$ points so the trace of Frobenius is $2 \mod{3}$, hence the image of Frobenius (with respect to $\{P_1,P_2\}$) is $\SL_2$-conjugate to $\left(\begin{smallmatrix}
			1 & n \\
			0 & 1
		\end{smallmatrix}\right)$ for some $n \in \{0,1,2\}$. These all define distinct $\SL_2$-conjugacy classes, with the non-identity elements being $\GL_2$-conjugate.
		
	A quick check shows that $\tilde{E}(\mathbb{F}_{13})[3] \neq 9$ hence $n \neq 0$ so $\left(\begin{smallmatrix}
			1 & 1 \\
			0 & 1
		\end{smallmatrix}\right)$ is $\GL_2$-conjugate to $\rho_{E,3}(\Frob_{\mathfrak{p}})$. 
		
	Now $\tilde{E}[3]$ is defined over the cubic extension $\mathbb{F}_{13}[\alpha]$, where $\alpha$ has minimal polynomial $x^3+2x-2$. We compute that $Q_1=(10,6)$, $Q_2=(8\alpha^2-\alpha+3,7\alpha^2+4\alpha-1)$ is a basis of $\tilde{E}[3]$ such that the Frobenius automorphism acts as $\left(\begin{smallmatrix}
			1 & 1 \\
			0 & 1
		\end{smallmatrix}\right)$ here.
	
	The criterion we have in this case is equivalent to checking whether $\langle P_1, P_2 \rangle_3 \equiv \langle Q_1,Q_2 \rangle_3 \mod{\mathfrak{p}}$. A quick calculation shows that $\langle Q_1,Q_2 \rangle_3 =3$ so $\rho_{E,3}(\Frob_{\mathfrak{p}})$ is $\SL_2$-conjugate to $\left(\begin{smallmatrix}
			1 & 1 \\
			0 & 1
		\end{smallmatrix}\right)$ with respect to $\{P_1,P_2\}$.
\end{example}

\begin{example}
	Consider the elliptic curve $y^2+y=x^3-x^2$ (Cremona label 11a3) defined over $\qq(\sqrt{5})$. The mod $5$ image is isomorphic to $D_{10}$, the dihedral group of order 10. This is not contained in $\SL_2(\mathbb{F}_5)$ but is contained in $\GL_2^{\square}(\mathbb{F}_5)$. 
	
	Let $\mathfrak{p}=(\frac{1+5\sqrt{5}}{2})$ be a prime of $\qq(\sqrt{5})$ above $31$. Choose the ordered global basis $P_1 \approx (1.69-1.54i,-1.27+2.83i), P_2=(1,-1)$ so $\langle P_1, P_2 \rangle_5 = e^{2\pi i/5}$. This is not an element of $\qq(\sqrt{5})$ so we instead take its minimal polynomial $m_{\qq(\sqrt{5})}(e^{2\pi i/5})=x^2+\frac{1}{2}(1+\sqrt{5})x+1$.
	
	One can check that $\Frob_{\mathfrak{p}}$ has order $5$ using the group structure of the reduced curve and so is conjugate to either $\left(\begin{smallmatrix}
			1 & 1 \\
			0 & 1
		\end{smallmatrix}\right)$ or $\left(\begin{smallmatrix}
			1 & 2 \\
			0 & 1
		\end{smallmatrix}\right)$ under the ordered basis $\{P_1,P_2\}$.
	
	Let $Q_1=(1,-1), Q_2=(26\alpha^3+8\alpha^2+23\alpha+12, 16\alpha^4+17\alpha^3+29\alpha^2+17\alpha+2)$, where $\alpha$ has minimal polynomial $x^5+7x+28$. Then the Frobenius automorphism acts on $\tilde{E}[5]$ as $\left(\begin{smallmatrix}
			1 & 1 \\
			0 & 1
		\end{smallmatrix}\right)$ with respect to the ordered basis $\{Q_1,Q_2\}$.
	
	We compute that $\langle Q_1,Q_2 \rangle_5 = 8$. Now $m_{\qq(\sqrt{5})}(e^{2\pi i/5}) \equiv x^2+13x+1 \mod{\mathfrak{p}}$ which doesn't have $8$ as a root so $\Frob_{\mathfrak{p}}$ cannot be conjugate to $\left(\begin{smallmatrix}
			1 & 1 \\
			0 & 1
		\end{smallmatrix}\right)$. Redoing the calculation with $\left(\begin{smallmatrix}
			1 & 2 \\
			0 & 1
		\end{smallmatrix}\right)$, (where we take the basis $\{Q_1,Q_1+2Q_2\}$), the Weil pairing is $2$ which is a root of $m_{\qq(\sqrt{5})}(e^{2\pi i/5}) \mod{\mathfrak{p}}$. Hence $\Frob_{\mathfrak{p}}$ is $D_{10}$-conjugate to $\left(\begin{smallmatrix}
			1 & 1 \\
			0 & 1
		\end{smallmatrix}\right)$ with respect to the basis $\{P_1,P_2\}$.
\end{example}

\begin{remark}
	In our second example, we could have base changed our curve to $\qq(\zeta_5)$ as the Frobenius element would be unchanged since $\mathfrak{p}$ splits. Even if $\mathfrak{p}$ was inert, we could also have dealt with that as we would just get the other conjugacy class instead.
\end{remark}

\begin{remark}
We ran our method against the current algorithm of Dokchitser and Dokchitser in Magma. Their algorithm is not yet implemented over number fields so we only ran ours for rational primes which were completely split in the base field so the Frobenius element is unchanged. In addition, the bulk of the computation in their method consists of constructing a polynomial for each conjugacy class first. For a fairer comparison, we chose to time the results to determine the Frobenius elements at 1000 suitable rational primes in the mod $3,5,7$ and $11$ representations of the elliptic curve $y^2=x^3+x+1$; the computation was run on a machine with an AMD Opteron(tm) Processor 6174 and a speed of 2200MHz. We tabulate our results below.

\begin{center}
\begin{tabular}{l|l|l}
$l$  & Weil Pairing Method & Dokchitsers' Method \\ \hline
$3$  & $5.7$ seconds     & $0.5$ seconds     \\
$5$  & $25.7$ seconds    & $11.4$ seconds    \\
$7$  & $88.7$ seconds    & $1032.3$ seconds  \\
$11$ & $373.4$ seconds   & $>7$ days             
\end{tabular}
\end{center}

\end{remark}

The final thing we wish to address is how beneficial elliptic curves were here as to the feasibility of this method for Galois representations arising from other types of objects. We can also do this for larger dimensional vector spaces, so we'll incorporate this into our theorem.

We first recall a construction which generalises the precise properties of the Weil pairing that we want. Let $V/\fl$ be a vector space of dimension $n$. Then the $n^{th}$ exterior power $\Lambda^nV^*$ is a one dimensional vector space of alternating multilinear forms, such that for any nonzero $T \in \Lambda^nV^*$ we have
\begin{enumerate}
	\item $T(v_1,\cdots, v_n)=0$ if and only if $\{v_1, \cdots, v_n\}$ are linearly dependent,
	\item $T(Av_1,\cdots, Av_n)=\det(A)T(v_1,\cdots,v_n)$ for all matrices $A \in \GL_n$.
\end{enumerate}

In the case of the Weil pairing, we identified the image $\fl$ with the $l^{th}$ roots of unity and shall do so again in our final theorem. For a field $F$, we let $\mu_l(F)$ denote the $l^{th}$ roots of unity in $F$.

\begin{theorem}
\label{genthm}
	Let $K/F$ be a Galois extension of number fields, such that $\rho: \Gal(K/F) \rightarrow  \SL_n(\fl)$ is an isomorphism for some rational prime $l$ and positive integer $n$. Let $\mathfrak{p}$ be a prime of $F$ which is unramified in $K$ and $\mathfrak{P}$ a prime of $K$ above $\mathfrak{p}$ with corresponding residue fields $\mathcal{F}$ and $\mathcal{K}$. Write $G=\Gal(K/F)$ and $\overline{G}=\Gal(\mathcal{K}/\mathcal{F})$, where we identify the latter with the decomposition subgroup.

	Let $V, \overline{V}$ be two $\fl$-vector spaces of dimension $n$. Suppose $V$ (respectively $\overline{V}$) has a faithful action of $G$ (respectively $\overline{G}$) and there exists an isomorphism $\theta: V \rightarrow \overline{V}$ such that $\theta \overline{g} = \overline{g} \theta$ for all $\overline{g} \in \overline{G}$. Furthermore, suppose that there are nonzero alternating multilinear forms $T_F \in \Lambda^n V^*$ and $T_{\mathcal{F}} \in \Lambda^n\overline{V}^*$ such that the diagram

\begin{center}
\begin{tikzcd}
V^n \arrow[r, "T_F"] \arrow[d,"\theta_n"] & \mu_l(F) \arrow[d, "\mod{\mathfrak{p}}"] \\
\overline{V}^n \arrow[r,  "T_{\mathcal{F}}"] & \mu_l(\mathcal{F})
\end{tikzcd}
\end{center}

	commutes, where $\theta_n(v_1,\cdots,v_n):=(\theta(v_1),\cdots,\theta(v_n))$.
	
	Suppose the $\GL_n(\fl)$-conjugacy class of $\rho(\Frob_{\mathfrak{p}})$ splits into $m$ classes in $\SL_n(\fl)$ and let $H \subset \fl^{\times}$ be the unique subgroup such that $[\fl^{\times}:H]=m$. Suppose $\sigma \in \SL_n(\fl)$ is $\GL_n(\fl)$-conjugate to $\rho(\Frob_{\mathfrak{p}})$ and let $\overline{B}$ be an ordered basis of $\overline{V}$ such that the Frobenius automorphism acts as $\sigma$ on $\overline{V}$ with respect to $\overline{B}$. Then $\rho(\Frob_{\mathfrak{p}})$, written with respect to a global ordered basis $B$, is $\SL_n(\fl)$-conjugate to $\sigma$ if and only if $$T_F(B) \, \bmod{\mathfrak{p}} \, \equiv T_{\mathcal{F}}(\overline{B})^h \qquad \text{ for some } h \in H.$$
\end{theorem}

\begin{proof}
	When $n=2$, the proof is the same as the proof of Theorem \ref{weilmethod} as $m \leq 2$. 
	
	Note that we used that the centraliser of $\sigma$ in $\GL_2$, $C_{\GL_2}(\sigma)$, was equal to its centraliser in $\GL_2^{\square}$ when the class splits; for general $n$, the latter group should be replaced by the group $\GL_n^{(n)}(\fl)$ of matrices whose determinant is an $n^{th}$ power. Moreover, we necessarily have $(\fl^{\times})^n \subset H$ and $|C_{\GL_n(\fl)}(\sigma)|=|C_{\GL_n^{(n)}(\fl)}(\sigma)|[H:(\fl^{\times})^n]$, so a similar computation then proves this more general version.
\end{proof}

\begin{Ack}
	I wish to thank Vladimir Dokchitser for many useful discussions and Tim Dokchitser for his comments. I would also like to thank both the University of Warwick and King's College London where this research was carried out.
\end{Ack}

\nocite{Bosma1997}

\bibliographystyle{amsalpha}
\bibliography{Frob}
\end{document}